\documentclass[12pt]{amsart}
\usepackage{etex}
\usepackage[all,cmtip]{xy}
\usepackage{amssymb, amsmath, amscd, dcpic, pictexwd}
\usepackage{hyperref}
\usepackage{graphicx}
\usepackage{pdfpages}
\usepackage[normalem]{ulem}

\numberwithin{equation}{section}
\newtheorem{thm}{Theorem}[section]
\newtheorem{question}[thm]{Question}
\newtheorem{cor}[thm]{Corollary}
\newtheorem{lem}[thm]{Lemma}
\newtheorem{prop}[thm]{Proposition}

\newtheorem{claim}[thm]{Claim}
\newtheorem{defn}[thm]{Definition}
\theoremstyle{definition}
\newtheorem{exm}[thm]{Example}
\newtheorem{rem}[thm]{Remark}

\newcommand{\cO}{\mathcal O}

\newcommand{\cU}{\mathcal U}

\newcommand{\cY}{\mathcal Y }

\newcommand{\bR}{\bf R}

\newcommand{\C}{\mathbb C}

\newcommand{\HH}{\mathbb{H}}
\renewcommand{\H}{\HH}

\newcommand{\PP}{\mathbb P}
\renewcommand{\P}{\PP}

\newcommand{\R}{\mathbb R}

\newcommand{\Z}{\mathbb Z}

\newcommand{\gh}{\mathfrak{h}}

\newcommand{\ra}{\rightarrow}

\newcommand{\bs}{\bigskip}

\newcommand{\Hom}{{\mbox{Hom}}}
\newcommand{\pr}{{\mbox{pr}}}
\newcommand{\Aut}{{\mbox{Aut~}}}

\newcommand{\Pic}{{\mbox{Pic}}}
\newcommand{\Res}{{\mbox{Res~}}}
\newcommand{\id}{{\mbox{id}}}

\newcommand{\kerr}{{\mbox{ker~}}}
\renewcommand{\ker}{\kerr}
\newcommand{\bl}{\hskip.1in}

\def\question#1{{}}

\newcommand{\Cx}{\mathbb{C}^\times}
\DeclareMathOperator*{\im}{Im}
\DeclareMathOperator*{\Ext}{Ext}

\newcommand{\lra}{\leftrightarrow}

\newcommand{\xra}{\xrightarrow}

\newcommand{\quash}[1]{}

\author{Jingyue Chen and Bong H. Lian}

\title{CY Principal Bundles over Compact K\"ahler Manifolds}
\begin{document}

\maketitle

\begin{abstract}
A CY bundle on a connected compact complex manifold $X$ was a crucial ingredient in constructing differential systems for period integrals in \cite{LY}, by lifting line bundles from the base $X$ to the total space. A question was therefore raised as to whether there exists such a bundle that supports the liftings of {\it all} line bundles from $X$, simultaneously. This was a key step for giving a uniform construction of differential systems for arbitrary complete intersections in $X$. In this paper, we answer the existence question in the affirmative if $X$ is assumed to be K\"ahler, and also in general if the Picard group of $X$ is assumed to be {discrete}. 
Furthermore, we prove a rigidity property of CY bundles if the principal group is an algebraic torus, showing that such a CY bundle is essentially determined by its character map. 
\end{abstract}

\tableofcontents

\pagenumbering{arabic}
\addtocounter{page}{0}

\section{Introduction}
\subsection{Background}
Let $G,H$ be complex Lie groups and $X$ be a connected compact complex manifold with a $G$-action. A $G$-equivariant principal $H$-bundle, denoted by $M\equiv(G,H-M{\buildrel\pi\over\ra} X)$, is a holomorphic principal $H$-bundle $M$ over $X$ equipped with an action
$$
G\times H\times M\ra M,~~(g,h,m)\mapsto gmh^{-1}.
$$
Since the $H$-action is assumed to be principal, the projection $\pi$ induces an isomorphism $M/H\simeq X$.

Given such an $M$, it is easy to show that there is an equivalence of categories
$$
\{\mbox{$G$-equiv. vector bundles on $X$}\}\leftrightarrow\{\mbox{$(G\times H)$-equiv. vector bundles on $M$}\}.
$$
Restricting this to line bundles, we get a natural isomorphism
$$
\Pic_G(X)\simeq\Pic_{G\times H}(M).
$$

Let $\chi\in\Hom(H,\C^\times)\equiv\widehat{H}$ be a holomorphic character of $H$, and $\C_\chi$ be the corresponding 1-dimensional representation.
Then $G\times H$ acts on $M\times\C_\chi$ as a $G$-equivariantly trivial bundle which is not necessarily $H$-equivariantly trivial. Composing with the Picard group isomorphism above one gets a canonical homomorphism
\begin{align*}
\lambda_M:\widehat{H}&\ra\Pic_{G\times H}(M)\ra\Pic_G(X)\\
\chi\quad&\mapsto \quad M\times\C_\chi~~\mapsto L_\chi:=(M\times\C_\chi)/H.
\end{align*}
We call this {\it the character map} of $M$.

This map allows us to describe the line bundles in the image of $\lambda_M$ together with their sections purely in terms of 1-dimensional representations of $H$. For example, one can show that there is a canonical $G$-equivariant isomorphism \cite{LY}
$$
\Gamma(X,L_\chi)\simeq\cO_\chi(M):=\{f\in\cO(M)|f(mh^{-1})=\chi(h)f(m),\forall m\in M\}.
$$

\begin{defn} [{\cite{LY}}]
We say that $M$ is a CY bundle if it admits a CY structure $(\C\omega,\chi)$. Namely, $\chi\in\widehat{H}$ is a holomorphic character of $H$, and $\omega$ is a $G$-invariant nowhere vanishing holomorphic top form on $M$ such that
$$
\Gamma_h\omega=\chi(h)\omega,~~~h\in H.
$$
\end{defn}

A prototype example of this definition is given by the following example due to Calabi.

\begin{exm}[\cite{Ca}, 1979]
Let $M:=K_X^\times$ be the complement of the zero section of $K_X$. Let $\omega=dz_w\wedge dw_1\wedge\cdots\wedge dw_d$, where $w$ is a local coordinate chart and $z_w$ is the coordinate induced by $w$ along the local fibers of $K_X$. Then $\omega$ is a globally defined CY structure on the bundle $(\Aut X,\,\C^\times -M\ra X)$ with $\chi=\id_{\C^\times}$.
\end{exm}

Let us mention a number of important applications of this notion. First, we note that CY structures, if  exist, can be classified by a coset of the kernel of the character map.

\begin{thm}[Classification of CY structures \cite{LY}]\label{classify CY structures} Given a principal $H$-bundle over a compact complex manifold $X$, there is a bijection
\[
\{\text{CY structures on $M$} \} 
\longleftrightarrow
 \lambda_{M}^{-1}([K_X]),~~(\C\omega,\chi)\lra\chi\chi_\gh  
\]
where $\gh$ denotes the Lie algebra of $H$ and $\chi_\gh$ is the 1-dimensional representation $\wedge^{\dim \gh}\gh$ induced by the adjoint representation of $H$.
\end{thm}

Next, if such a structure exists, one gets a bundle version of the adjunction formula. It allows us to describe $K_X$ in purely functional terms.

\begin{thm}[Adjunction for bundle \cite{LY}]\label{adjunction} 
Let $(\C\omega,\chi)$ be a CY structure on $M$. Then there is a canonical isomorphism
$$
K_X\simeq L_{\chi\chi_\gh}.
$$
\end{thm}
As a consequence, we have the following corollary:
\begin{cor}
There is a canonical embedding of the pluri-(anti)canonical ring of $X$ as a subring of the ring of holomorphic functions $\cO(M)$.
\end{cor}

\subsection{Motivations and Applications}
The aim of this paper is to study the existence and uniqueness questions for CY bundles. We are particularly interested in those CY bundles whose character map $\lambda_M$ is surjective. In this case, {\it all line bundles on $X$ can be simultaneously realized by $H$-characters.}

This was an important open question raised in \cite{LY}. Such a structure would fill a crucial step in the construction of tautological systems for period integrals of arbitrary complete intersections in $X$. In addition, it provides a uniform treatment for all line bundles on $X$.

\bs

Let $X$ be a  connected $d$-dimensional compact complex manifold, and $L_1,\ldots, L_s$ be line bundles on $X$ such that $V_i:=H^0(X,L_i)^*\neq 0$
and that the general section $\sigma_i\in V_i^*$ defines a nonsingular hypersurface $Y_{\sigma_i}=\{\sigma_i=0\}$ in $X$. Consider the family of all complete intersections $Y_\sigma:=Y_{\sigma_1}\cap\cdots\cap Y_{\sigma_s}$ which are smooth of codimension $s <d$ in $X$. Put $V:=V_1\times\cdots\times V_s$ and $B=V^*-D$ where $D$ consists of $\sigma\in V^*$ such that $Y_\sigma$ is not smooth of codimension $s$. Then $B$ parametrizes a smooth family $\cY$ of smooth complete intersections of codimension $s$ in $X$. Given $\sigma\in B$, the adjunction formula gives a canonical isomorphism
\[(L+K_X)|{Y_\sigma}\simeq K_{Y_\sigma}\]
where $L:=L_1+\cdots+L_s$. Thus when $L+K_X$ is trivial, $\cY$ is a family of CY manifolds, and when $L+K_X$ is ample, $\cY$ is a family of general type manifolds.

We shall assume that the Hodge number $\dim H^0(Y_\sigma,K_{Y_\sigma})$ is a locally constant function on $B$ (this is true when the fibers of $\cY$ are K\"ahler), and consider bundle $\H^\text{top}$ whose fiber at $\sigma\in B$ is $H^0(Y_\sigma,K_{Y_\sigma})$. For $\sigma\in B$, let
\[{\bR}_\sigma: H^0(X,L+K_X)\ra H^0(Y_\sigma,K_{Y_\sigma}) \]
be the restriction map. Then for each $\tau\in H^0(X,L+K_X)$, the map $$B\ra \H^\text{top},\quad \sigma\mapsto {\bR}_\sigma(\tau)$$
defines a global section of $\H^\text{top}$. So we get a linear map
\[{\bR}: H^0(X,L+K_X)\ra H^0(B, \H^\text{top}), \quad \tau\mapsto(\sigma\mapsto {\bR}_\sigma(\tau)).\]
In particular when $L+K_X{=\cO_X}$, then ${\bR}(1)$ gives a canonical global trivialization of $ \H^\text{top}$.

\begin{defn}[Period sheaves \cite{LY}]
We call the map {\bR} the global Poincar\'e residue map of the family $\cY$. For each $\tau\in H^0(X,L+K_X)$, we define the period sheaf ${\bf \Pi}(\tau)$ of the family $\cY$ to be the locally constant sheaf on $B$ generated by the local sections $$\int_\gamma  {\bR}(\tau),\; \gamma\in H_{d-s}(Y_\bullet,\Z)$$
where $Y_\bullet$ is some fixed fiber of $\cY$. A local section of this sheaf is called a period integral.
\end{defn}

Let $H-M\ra X$ be a bundle with CY structure $(\omega_M,\chi_M)$.
Fix a complex Lie subgroup $G\subset\Aut X$. 
If we assume further that $L_1,\ldots, L_s$  are  $G$-equivariant line bundles on $X$ such that $H^0(X, L+K_X)\neq 0$, $V_i^*:=H^0(X,L_i)$ is base point free, and there is an $H$-character $\chi$ such that $L\simeq L_\chi.$
(We note that if the character map $\lambda_M:\widehat{H}\ra\Pic_G(X)$ is onto, then the last assumption  is redundant.)
Then there is an explicit formula for the family version of the Poincar\'e residue map for complete intersections:
\begin{thm}[Global Poincar\'e residue \cite{LY}]
Let $L_1,\ldots, L_s$ be assumed as above. Fix independent vector fields $x_i$ generated by $H$ on $M$. There is a constant $c$ such that, for $\tau\in H^0(X,L+K_X),\; \sigma\in B,$ and $\psi:=\sigma_1\cdots\sigma_s\in H^0(X,L_\chi)$, we have
\[{\bR}_\sigma(\tau)=c\Res\frac{\tau_M}{\psi_M}\iota_{x_1}\cdots \iota_{x_q}\omega_M.\]
\end{thm}

This provides a powerful tool for studying Picard-Fuchs systems for period integrals and their solutions. This technique can reconstruct virtually all known Picard-Fuchs systems, plus a large class of new ones (tautological systems). See recent papers \cite{LSY},\cite{LY} and \cite{HLZ}.

CY bundles are also very useful for studying the $D$-modules associated with the Picard-Fuchs systems. A CY bundle allows us to describe the $D$-modules in terms of Lie algebra homology. In some important cases, this can be recast as the de Rham cohomology of an affine algebraic variety via the Riemann-Hilbert correspondence. See recent papers  \cite{BHLSY}, \cite{HLZ} and \cite{HLYZ}.

\subsection{Existence and uniqueness problems for CY bundles} 

For simplicity, we will assume that $G=1$. Every result we discuss has a $G$-equivariant version, where $G$ is a suitable lifting of any closed subgroup of $\Aut X$.

\bs
We begin with the special case $H\simeq(\C^\times)^p$ for a positive integer $p$, so that $\chi_\gh$ is trivial. 
The following result says that a CY bundle is uniquely determined by its character map.

\begin{thm} [Rigidity of CY bundles]
For $H=(\C^\times)^p$, let $M_i\equiv(H-M_i\ra X)$, $i=1,2$, be CY bundles such that the following diagram  
\[\xymatrix{
\widehat{H}\ar[r]^{\lambda_{M_1}\quad}\ar[d]^{\simeq}_{\xi}&\Pic(X)\\
\widehat{H}\ar[ur]_{\lambda_{M_2}}
}\]
commutes for some $\xi\in\Aut \widehat{H}$. Then there is an isomorphism $M_1\simeq M_2$, canonical up to a twist  by the induced automorphism $\xi^\vee\in\Aut H$.
\end{thm}
{We can show further that the character map classifies holomorphic principal $(\C^\times)^p$-bundles.}

This result will be proved in Section \ref{sect-rigidity}.

\bs
We now consider the existence question. The first crucial step is the following criterion, which can also be seen as a consequence of Theorem \ref{classify CY structures}.

\begin{thm} [Obstruction criterion\cite{LY}]\label{obstruction thm}
An $H$-principal bundle $M$ admits a CY structure iff $K_X$ is in the image of the character map $\lambda_M$.
(By adjunction for bundles, we necessarily get $K_X=\lambda_M(\chi\chi_\gh)=L_{\chi\chi_\gh}$.)
\end{thm}

This gives a classification for rank 1 CY bundles.

\begin{cor}
The bundle $M\equiv(\C^\times-M\ra X)$ admits a CY structure iff $M$ is the complement of the zero section of a line bundle $L$ which is a root of $K_X$, i.e. $K_X\simeq kL$ for some integer $k$.
\end{cor}

\begin{exm}
Take $X=\P^d$, $H=\C^\times$. Then rank 1 CY bundles on $X$ are exactly those of the form $M\simeq\cO(k)^\times$ with the prescribed $\Cx$-action $h\cdot l:=lh^{-1}, h\in\Cx,l\in M$ for some $k|(d+1)$. Its character map is then
$$
\lambda_{M}:\widehat{H}=\Z\ra\Pic(X)\simeq\Z,~~~1\mapsto k.
$$
Thus $\lambda_{M}$ is isomorphic iff $k=\pm1$. 
Let $M_1:=\cO(-1)^\times, M_2:=\cO(1)^\times$, then $\xi: \widehat{H}\ra \widehat{H}, 1\mapsto -1$ is an isomorphism and $\lambda_{M_1}=\lambda_{M_2}\circ\xi.$
So by the rigidity theorem,  $\cO(-1)^\times\simeq\cO(1)^\times$ up to a twist by $\xi^\vee: H\ra H,~~h\mapsto h^{-1}$.


We can describe the isomorphism explicitly as follows. Let $N:=\C^{d+1}\setminus\{0\}$ with a $\Cx$-action $h\cdot m=mh^{-1}$. Since $\cO(-1)\subset \P^d\times \C^{d+1}$, we have an isomorphism of fiber bundles $\alpha: N\ra\cO(-1)^\times, m\mapsto([m],m)$. 
Since $h\cdot([m],m)=([m],mh^{-1})=([mh^{-1}],mh^{-1})=\alpha(mh^{-1})=\alpha(h\cdot m)$, $\alpha$ is an isomorphism of principal $\Cx$-bundles between $M_1$ and $N$.

We can define a linear function $m^{-1}:\cO(-1)_{[m]}\ra\C, ([m],cm)\mapsto c$. 
Then $m^{-1}\in (\cO(-1)_{[m]})^\vee\simeq \cO(1)_{[m]}$. Let $\beta:N\ra\cO(1)^\times, m\mapsto([m],m^{-1})$. Then $\beta$ is an isomorphism of fiber bundles as well. However, $h\cdot ([m],m^{-1})=([m],m^{-1}h^{-1})=([mh],(mh)^{-1})=\beta(mh)=\beta(h^{-1}\cdot m)$. This implies that as principal bundles, $M_2$ is isomorphic to the bundle $N$ but with $\Cx$-action twisted by $\xi^\vee$. Therefore $M_1$ and $M_2$ are isomorphic as principal bundles up to a twist by $\xi^\vee.$

\end{exm}

This example generalizes to any smooth toric variety. As an application, we can give a simple characterization of an important object in toric geometry constructed by Audin and Cox in the early 90's. Namely, let $T=(\C^\times)^d$ and $X$ be a smooth complete toric variety with respect to the group $T$.

\begin{thm} [Audin-Cox variety]
Let $t$ be the number of $T$-divisors in $X$. Then
there is a canonical  $(\C^\times)^t$-invariant Zariski open subset $M\subset\C^t$ and a $(t-d)$-dimensional closed algebraic subgroup $H\subset(\C^\times)^t$ such that the geometric quotient $M/H$ is isomorphic to $X$.
\end{thm} 

It can be shown that the character map $\lambda_M$ for the Audin-Cox variety is  an isomorphism. In particular $M$ is a CY bundle over $X$. The $\P^d$ example above is a special case of this construction. Now as a consequence of our rigidity theorem, we have the following characterization of $M$:

\begin{cor}
The Audin-Cox variety $M$ is the unique (up to a twist) CY bundle over $X$ with the property that $\lambda_M$ is a group isomorphism.
\end{cor}

Next, we will see that this allows us to reconstruct the same space $M$ in many different ways as an algebraic variety. The characterization also shows that a CY bundle over a general complex manifold can be viewed as a  generalization of the Audin-Cox construction for toric varieties. 

Here is one of our main results.

\begin{thm} {} 
Let $X$ be a compact complex manifold. If $\Pic(X)$ is free, then $X$ admits a unique (up to a twist) CY bundle whose character map is an isomorphism. If $\Pic(X)$ is finitely generated, then $X$ admits a CY bundle whose character map is an isomorphism.
If $X$ is K\"ahler, then it admits a CY bundle whose character map is onto.
\end{thm}

These results will be proved in Sections \ref{sect-rigidity}, \ref{sect-Pic} and \ref{sect-discrete}.

When $\Pic(X)$ is free, the proof uses the obstruction criterion of \cite{LY}, and the rigidity theorem above. When $\Pic(X)$ is not free, then by using the  K\"ahler condition, we can lift the construction to the universal cover $\tilde X$, where $\Pic(\tilde X)$ is finitely generated. Then  we apply a construction similar to the first case. 

Moreover, making use of the Remmert-Morimoto decomposition for connected abelian complex Lie groups, we have:
\begin{thm}
If $X$ is compact K\"ahler and $H$ is a sufficiently large connected abelian group, then there exists a CY $H$-bundle whose character map is onto.
\end{thm}


This result will be proved in Section \ref{sect-abelian}.

\bs

\section{Existence of $(\Cx)^p$-principal bundles}\label{sect-Cx}

Recall that a holomorphic principal $H$-bundle, where $H$ denotes a complex Lie group, is a holomorphic bundle $\pi: M\ra X$ equipped with a holomorphic right action $M\times H\ra M$, such that $H$ acts on each fiber of $\pi$ freely and transitively. We denote the principal bundle as $H-M\ra X$. If we have two principal bundles $H_1-M_1\ra X$ and $H_2-M_2\ra X$, then the Whitney sum $M_1\oplus M_2$ is a principal $(H_1\times H_2)$-bundle over $X$. 

Let $X$ be a complex manifold, and $L$ be a holomorphic line bundle over $X$. Let $L^{\times}$ denote the complement of the zero section of $L$. Then we have a natural $\Cx$-action on $L^{\times}$:
\[\Cx\times L^{\times}\ra L^{\times}, (h, l)\mapsto lh^{-1}\]
where $h\in\Cx, l\in L^{\times}$.

It is clear that this action preserves each fiber of $L^\times\ra X$ and it is free and transitive, which means that the projection $L^{\times}\ra X$ defines a principal $\Cx$-bundle over $X$.

Now if we have two holomorphic line bundles $L_1, L_2$ on $X$, then $L_1^{\times}\oplus L_2^{\times}\subset L_1\oplus L_2$ is a principal $(\Cx)^2$-bundle over $X$ with the $(\Cx)^2$-action given by
\[(h_1,h_2)\cdot(l_1,l_2)=(l_1h_1^{-1},l_2h_2^{-1}).\]
We can represent any given $(\Cx)^2$-character $\chi\in \widehat{(\Cx)^2}=\Z^2$ uniquely as a product
\[\chi_k\chi_l: (\Cx)^2\ra\Cx, (h_1,h_2)\mapsto h_1^kh_2^l\]
for some $k,l\in\Z$, where $\chi_k\in\widehat{\C^\times}$.

\bigskip

\begin{prop}\label{1st-prop}
Given holomorphic line bundles $L_1,L_2$ on $X$, we have an isomorphism of holomorphic line bundles:
\[(L_1^\times \oplus L_2^\times) \times_{(\mathbb{C}^\times)^2}\mathbb{C}_{\chi_k\chi_l}\simeq kL_1+lL_2 .\]
\end{prop}

\begin{proof}
Define a map
\begin{align*}
\rho: (L_1^\times \oplus L_2^\times) \times \C_{\chi_k\chi_l} &\rightarrow kL_1+lL_2,\\
                         ((l_1, l_2), c)              &\mapsto c l_1^{\otimes k}\otimes l_2^{\otimes l}.
\end{align*}
Since $(\Cx)^2$ acts on $(L_1^\times \oplus L_2^\times) \times \C_{\chi_k\chi_l} $ by:
\[(h_1,h_2)\cdot ((l_1, l_2), c) =((l_1h_1^{-1},l_2h_2^{-1}),c h_1^kh_2^l),\] 
we have
\[\rho((h_1,h_2)\cdot ((l_1, l_2), c))=c h_1^kh_2^l (l_1h_1^{-1})^{\otimes k}\otimes (l_2h_2^{-1})^{\otimes l }=\rho  (((l_1, l_2), c))\]
 for any $(h_1, h_2)\in (\Cx)^2$. It follows that $\rho$ descends to
\[\tilde{\rho}: (L_1^\times \oplus L_2^\times) \times_{(\mathbb{C}^\times)^2}\mathbb{C}_{\chi_k\chi_l} \rightarrow kL_1+lL_2.\]
It is clear that $\tilde{\rho}$ induces a linear isomorphism on fibers and   commutes with quotient maps to $X$, thus it is an isomorphism.
\end{proof}

\begin{cor}\label{2nd-prop}

Let $L_1, L_2$ be line bundles and $k,l\in\Z$ such that $K_X\simeq kL_1+lL_2$. Then $L_1^\times \oplus L_2^\times$ admits a CY $(\Cx)^2$-bundle structure with character $\chi_k\chi_l$.
\end{cor}

\begin{proof}
Put $H:=(\Cx)^2, M:=L_1^\times \oplus L_2^\times, \chi:=\chi_k\chi_l$.
Proposition \ref{1st-prop} tells us that
\[K_X\simeq kL_1+lL_2\simeq M\times_H \C_{\chi}=\lambda_M(\chi).\]
By the obstruction criterion, i.e. Theorem \ref{obstruction thm}, $L_1^\times \oplus L_2^\times$ admits a CY $H$-bundle structure $(\C\omega_M,\chi_M)$.
Moreover, since the $H$-character $\chi_\mathfrak{h}$ is trivial, $\chi_M=\chi\chi_{\mathfrak{h}}^{-1}=\chi$ by Theorem \ref{adjunction}.
\end{proof}

\begin{cor}
For every line bundle $L$, there exists a CY bundle $(\Cx)^2-M\ra X$ and a character $\chi\in\widehat{(\Cx)^2}$ such that $L\simeq M\times_{(\Cx)^2} \C_{\chi}$. 
\end{cor}

\begin{proof}

Let $M:=L^\times\oplus K_X^\times $. Since $K_X=0\cdot L+1\cdot K_X$, Corollary \ref{2nd-prop} tells us that $M$ is a CY bundle with character $\chi_0\chi_1$. Now for $\chi=\chi_1\chi_0$, we have $M\times_{(\Cx)^2} \C_{\chi}\simeq 1\cdot L+0\cdot K_X=L$.
\end{proof}

\begin{exm}
Since $K_{\P ^n}=\mathcal{O}(-n-1)$, $K_{\P ^n}\simeq k\mathcal{O}(-1)+l\mathcal{O}(-1)$ whenever $k+l=n+1$ for $k,l\in\Z$. Corollary \ref{2nd-prop} shows that $M:=\mathcal{O}(-1)^\times\oplus\mathcal{O}(-1)^\times$ is a CY bundle over $\P^n$. We now give an explicit description of the CY structure $(\C\omega_M, \chi_M)$. 

We have just seen that $\chi_M=\chi_k\chi_l$.
We can take
$$
M=\{([m],m,m')| m\in\C^{n+1}\setminus\{0\},~ \C m=\C m'\}\subset\P^n\times\C^{n+1}\times\C^{n+1}.$$ 
Then we have a canonical isomorphism
$$
M\simeq (\C^{n+1}\setminus\{0\})\times\C^\times,~~([m],m,c m)\lra (m,c).
$$
The $H$-action then becomes $h\cdot(m,c)=(mh_1^{-1},ch_1h_2^{-1})$ for $h=(h_1,h_2)\in H$.
On the right hand side, we have the global coordinates $(z,\zeta)\equiv(z_0,..,z_n,\zeta):(m,c)\mapsto(m_0,..,m_n,c)$. 
We then find that the CY structures on $M$ are just $(\C\omega_l,\chi_k\chi_l)$, $l\in\Z$, where
$$
\omega_l:=\zeta^{l-1}dz_0\wedge\cdots\wedge dz_n\wedge d\zeta,\bl\bl\chi_{k}\chi_l(h)=h_1^kh_2^l.
$$
\end{exm}

\bs

We can generalize Proposition \ref{1st-prop} and Corollary \ref{2nd-prop} to cases involving finitely many line bundles:

\begin{thm}\label{discrete}
Given holomorphic line bundles $L_1,\cdots,L_p$ on $X$, we have an isomorphism of holomorphic line bundles:
\[(L_1^{\times}\oplus L_2^{\times}\oplus \cdots \oplus L_p^{\times})\times_{(\Cx)^p}\C_{\chi_{k_1}\cdots\chi_{k_p}}\simeq k_1L_1+\cdots+k_pL_p \]
where $\chi_{k_1}\cdots\chi_{k_p}: (\Cx)^p\rightarrow\Cx, (h_1,\cdots,h_p)\mapsto h_1^{k_1}\cdots h_p^{k_p}$.
\end{thm}

\begin{cor}
If there exist line bundles $L_1, \cdots, L_p$ and integers $k_1,\cdots, k_p$ such that $k_1L_1+k_2L_2+\cdots+k_pL_p\simeq K_X$, then $L_1^{\times}\oplus L_2^{\times}\oplus \cdots \oplus L_p^{\times}$ is a CY $(\Cx)^p$-bundle over $X$ with character $\chi_{k_1}\chi_{k_2}\cdots\chi_{k_p}$.
\end{cor}
\bs

\section{Rigidity of $(\Cx)^p$-principal bundles}\label{sect-rigidity}
{
First we discuss a decomposition property of certain principal bundles.
\begin{prop}
Suppose $M$ is a holomorphic principal $(H_1\times H_2)$-bundle over $X$. Then $H_1-M/{H_2}\ra X$ and $H_2-M/{H_1}\ra X$ are holomorphic principal bundles and there is an isomorphism of holomorphic principal bundles:
\[M\simeq (M/{H_2})\oplus (M/{H_1}).\]
\end{prop}

\begin{proof}
Denote $p:M\ra X$ the projection map. 
A general element in $(M/{H_2})\oplus (M/{H_1})$ is of the form $(m_1H_2,m_2H_1)$ where $m_1,m_2\in M$ and $p(m_1)=p(m_2)$.
Given $(h_1,h_2)\in H_1\times H_2,$ the action is
$$(h_1,h_2)\cdot (m_1H_2,m_2H_1)=((m_1h_1^{-1})H_2,(m_2h_2^{-1})H_1).$$

Define $$\varphi: M\ra  (M/{H_2})\oplus (M/{H_1}),\quad m\mapsto (mH_1,mH_2).$$
Given $h_i\in H_i$ for $i=1,2$, then $h_1h_2=h_2h_1\in H_1\times H_2.$
Then 
\begin{align*}
\varphi((h_1h_2)m)&=((mh_1^{-1}h_2^{-1})H_2,(mh_2^{-1}h_1^{-1})H_1)\\
&=((mh_1^{-1})H_2,(mh_2^{-1})H_1)=(h_1,h_2)\cdot \varphi(m).
\end{align*}
Thus $\varphi$ is $(H_1\times H_2)$-equivariant.

Let $q$ be the projection $(M/{H_2})\oplus (M/{H_1})\ra X,$
then $q(m_1H_2,m_2H_1)=p(m_1)=p(m_2)$. In particular $q(\varphi(m))=q(mH_1,mH_2)=q(m)$. Thus $\varphi$ commutes with projection maps.

Therefore $\varphi$ is a principal morphism over $X$. It is showed on \cite[p.43]{Hu} that every principal morphism over $X$ is an isomorphism for continuous principal bundles. Following the proof of \cite{Hu}, we see that $\varphi$ is injective and $\varphi^{-1}$ is continuous. Since $\varphi$ is also holomorphic,  we can further conclude that  $\varphi^{-1}$ is holomorphic (Proposition on \cite[p.19]{GH}) and thus $\varphi$ is biholomorphic. Therefore $\varphi$ is an isomorphism of holomorphic principal bundles.
\end{proof}

\begin{cor}
Suppose $M$ is a holomorphic principal $(H_1\times\cdots\times H_p)$-bundle over $X$. Let $M_i:=M/(H_1\times \cdots \times H_{i-1}\times H_{i+1}\times\cdots\times H_p)$. Then $M_i$ is a principal $H_i$-bundle over $X$. And $M$ admits a Whitney  sum decomposition: 
\[M\simeq M_1\oplus\cdots\oplus M_p.\]
\end{cor}

In particular, if $M$ is a principal $(\Cx)^p$-bundle over $X$. Then $M$ admits a Whitney sum decomposition:
\[M\simeq M_1\oplus\cdots\oplus M_p\]
where $M_i$ is a $\Cx$-bundle over $X$ for all $i$. By Lemma 3.7 in \cite{LY}, there exist line bundles $L_i\in\Pic(X)$ such that $M_i\simeq L_i^\times$ as principal $\Cx$-bundles for all $i$. Thus we have an isomorphism of holomorphic principal $(\Cx)^p$-bundles
\[M\simeq L_1^\times\oplus\cdots\oplus L_p^\times,\]
where the $(\Cx)^p$-action on the right hand side is
\[(h_1,\ldots,h_p)\cdot(l_1,\ldots,l_p)=(l_1h_1^{-1},\ldots, l_ph_p^{-1})\]
for $h_i\in\Cx,\; l_i\in L_i^\times.$

Our idea is to first decompose a given $(\Cx)^p$-bundle in terms of a Whitney sum, and then compare the  actions of different $(\Cx)^p$-bundles. 
\bs
}


For the rest of this section we set $H:=(\Cx)^p$. Assume $M$, $N$ are two $H$-principal bundles over a complex manifold $X$. Then we have character maps:
\begin{align*}
\lambda_M: \widehat{H}&\ra \Pic(X)\\
\gamma& \mapsto L^M_\gamma:=M\times_H \C_{\gamma}
\end{align*}
and
\begin{align*}
\lambda_N: \widehat{H}&\ra \Pic(X)\\
\rho& \mapsto L^N_\rho:=N\times_H \C_{\rho}
\end{align*}
where $\gamma, \rho\in \widehat{H}$.

We assume that there is an automorphism $\xi\in \Aut(\widehat{H})$ such that the following diagram\[\xymatrix{
\widehat{H}\ar[r]^{\lambda_M\quad}\ar[d]^{\simeq}_{\xi}&\Pic(X)\\
\widehat{H}\ar[ur]_{\lambda_N}
}\]
commutes.

Since 
$$\widehat{H}=\widehat{(\Cx)^p}=(\widehat{\Cx})^p\simeq\Z^p,$$
we can pick a $\Z$-basis $\{\gamma_1,\ldots,\gamma_p\}$ of $\widehat{H}$. Let $\rho_i:=\xi(\gamma_i),$ since $\xi$ is an automorphism, $\{\rho_1,\ldots,\rho_p\}$ is also an $\Z$-basis of $\widehat{H}$. Moreover, the commutative diagram tells us that there exists a holomorphic line bundle isomorphism $\eta_i:L^M_{\gamma_i}\ra L^N_{\rho_i}$.

Let $[m, c]^M_{\gamma_i}$ denote the class of $(m,c)\in M\times\C$ in $L^M_{\gamma_i}.$

\begin{lem}\label{construction of H-bundle}
The bundle $(L_{\gamma_1}^M)^\times\oplus\cdots\oplus(L_{\gamma_p}^M)^\times$ with $H$-action
\[h\cdot ([m_1,c_1]^M_{\gamma_1},\ldots, [m_p,c_p]^M_{\gamma_p})=([m_1h^{-1},c_1]^M_{\gamma_1},\ldots, [m_ph^{-1},c_p]^M_{\gamma_p})\]
is a principal $H$-bundle over $X$.
\end{lem}

\begin{proof}
It is clear that $(L_{\gamma_1}^M)^\times\oplus\cdots\oplus(L_{\gamma_p}^M)^\times$ is a bundle over $X$ with fiber $(\C^\times)^p=H$ and the $H$-action preserves the fiber. Since $[m_ih^{-1},\gamma(h)c_i]=[m_i,c_i]$, we have $[m_ih^{-1},c_i]=[m_i,\gamma(h^{-1})c_i]$. We can rewrite the $H$-action as
\[h\cdot ([m_1,c_1]^M_{\gamma_1},\ldots, [m_p,c_p]^M_{\gamma_p})=([m_1,\gamma_1(h^{-1})c_1]^M_{\gamma_1},\ldots, [m_p,\gamma_p(h^{-1})c_p]^M_{\gamma_p}) .\]

\begin{claim}
$\hat{\gamma}:=(\gamma_1,\ldots,\gamma_p)$ gives an automorphism of $H$. 
\end{claim}

\begin{proof}[Proof of claim] It is clear that $\hat{\gamma}$ is a homomorphism of $H$. Note that $\Aut(H)\simeq \text{GL}_p(\Z)$.
Since $\gamma_i\in\widehat{H}\simeq \Z^p$, there exists $a^i_1,\ldots, a^i_p\in\Z$ such that 
\begin{align*}
\gamma_i: H&\ra \C^\times\\
(h_1,\ldots,h_p)&\mapsto h_1^{a^i_1}\cdots h_p^{a^i_p}.
\end{align*}
In this way we may represent $\gamma_i$ by $(a^i_1,\ldots, a^i_p)$, then $\hat{\gamma}$ can be represented by the matrix $(a^i_j)$. Since $\{\gamma_1,\ldots,\gamma_q\}$ is a $\Z$-basis of $\widehat{H}$, we have $\det(a^i_j)=1$ and $(a^i_j)\in \text{GL}_p(\Z)$. Thus $\hat{\gamma}\in \Aut(H)$.
\end{proof}

Thus $H$ acts freely and transitively on the fiber of $(L_{\gamma_1}^M)^\times\oplus\cdots\oplus(L_{\gamma_p}^M)^\times$, i.e., it is a principal $H$-bundle.
\end{proof}

Similarly, $(L^N_{\rho_1})^\times\oplus\cdots\oplus(L^N_{\rho_p})^\times$ is a principal $H$-bundle with the action
\[h\cdot ([n_1,d_1]^M_{\rho_1},\ldots, [n_p,d_p]^M_{\rho_p})=([n_1h^{-1},d_1]^M_{\rho_1},\ldots, [n_ph^{-1},d_p]^M_{\rho_p}).\]

\begin{claim}\label{Misom}
The map
\begin{align*}
\alpha: M & \ra (L_{\gamma_1}^M)^\times\oplus\cdots\oplus(L_{\gamma_p}^M)^\times\\
m&\mapsto ([m,1]^M_{\gamma_1},\ldots, [m,1]^M_{\gamma_p})
\end{align*}
is an isomorphism of principle $H$-bundles.
\end{claim}

\begin{proof}Since
\[\alpha(h\cdot m)=([m_1h^{-1},1]^M_{\gamma_1},\ldots, [m_ph^{-1},1]^M_{\gamma_p})=h\cdot \alpha(m),\]
$\alpha$ is $H$-equivariant.

Let $p: M\ra X$ and $p':(L_{\gamma_1}^M)^\times\oplus\cdots\oplus(L_{\gamma_p}^M)^\times \ra X$ denote the projections respectively. Then $p'\circ \alpha(m)=p'(([m,1]^M_{\gamma_1},\ldots, [m,1]^M_{\gamma_p}))=p(m).$ Thus $\alpha$ is a principal morphism over $X$. {
Therefore $\alpha$ is an isomorphism of holomorphic principal $H$-bundles over $X$.}
\end{proof}

Similarly, let $\hat{\rho}:=(\rho_1,\ldots,\rho_p)$, then $\hat{\rho}\in \Aut(H)$, and we have an isomorphism of principal $H$-bundles:
\[\beta: N  \ra (L^N_{\rho_1})^\times\oplus\cdots\oplus(L^N_{\rho_p})^\times.\]

\bigskip

Now we can see that by definition $(L_{\gamma_1}^M)^\times\oplus\cdots\oplus(L_{\gamma_p}^M)^\times\simeq (L^N_{\rho_1})^\times\oplus\cdots\oplus(L^N_{\rho_p})^\times$ as holomorphic $H$-bundles, but the $H$-actions are different.

\begin{defn}
Given $\sigma\in \Aut(H)$ and a principal $H$-bundle $H-M\ra X$. Define $M^\sigma$ to be the bundle twisted by $\sigma$ as follows.
The total space of $M^\sigma$ is $M$ as a complex manifold, but the $H$-action on $M^\sigma$ is defined to be
\begin{align*}
H\times M^\sigma&\ra M^\sigma\\
(h,m)&\mapsto m(\sigma(h)^{-1}).
\end{align*}
\end{defn}

It is clear from the definition that the $H$-action on $M^\sigma$ preserves the fiber and acts freely and transitively on the fiber, so it is a principal $H$-bundle as well.

\begin{defn}Given principal bundles $H-M_i\ra X, i=1,2$,
we call a holomorphic bundle map $\varphi:M_1\ra M_2$ a twisted principal $H$-bundle isomorphism if there exists $\sigma\in \Aut(H)$ such that $\varphi:M_1\ra M_2^\sigma$ is a principal $H$-bundle isomorphism. In this case, we say that $M_1$ is isomorphic to $M_2$ up to a twist by $\sigma$.
\end{defn}

\begin{thm}[Rigidity of CY bundles]\label{unique}
For $H=(\Cx)^p$, assume $M$, $N$ are two $H$-principal bundles over a complex manifold $X$ and there  exists an automorphism $\xi\in \Aut(\widehat{H})$ such that the following diagram\[\xymatrix{
\widehat{H}\ar[r]^{\lambda_M}\ar[d]^{\simeq}_{\xi}&\Pic(X)\\
\widehat{H}\ar[ur]_{\lambda_N}
}\]
commutes. Then
$M$ is isomorphic to $N$ up to a twist by  the induced automorphism $\xi^\vee\in\Aut H$.
\end{thm}

\begin{proof}
First we describe the induced automorphism $\xi^{\vee}.$ It is clear that there is a natural homomorphism
\begin{align*}
\theta: \Aut H&\ra \Aut(\widehat{H})\\
f&\mapsto f^\vee:=(\gamma\mapsto \gamma\circ f^{-1})
\end{align*}
for $f\in \Aut H$ and $\gamma\in \widehat{H}$. The inverse of this map does not necessarily exist for general $H$. But in the case $H=(\Cx)^p$, 
$$\Aut H\simeq \text{GL}_p(\Z)\simeq \Aut(\widehat{H}),$$ we can construct an inverse of $\theta$.

\begin{claim}
Given a basis $\{\gamma_1,\ldots,\gamma_p\}$ of $\widehat{H}$, let $\hat{\gamma}:=(\gamma_1,\ldots,\gamma_p)$ and $\hat{\rho}:=(\xi({\gamma_1}),\ldots,\xi({\gamma_p}))$. Then 
\begin{align*}
\tau:  \Aut(\widehat{H})&\ra\Aut H\\
\xi&\mapsto \xi^\vee:=\hat{\rho}^{-1}\circ\hat{\gamma}
\end{align*}
is an inverse of $\theta.$
\end{claim}
\begin{proof}[Proof of claim:]
It is clear that $\xi^\vee\in\Aut H$. We want to show $(\theta\circ\tau)(\xi)=\xi$, it suffices to show that on the $\Z$-basis we have: $((\theta\circ\tau)(\xi))(\gamma_i)=\xi(\gamma_i)=\rho_i$.

Let $\pr_i:H\ra\Cx$ be the $i$-th projection, then $\gamma_i=\pr_i\circ \hat{\gamma}$ and $\rho_i=\pr_i\circ \hat{\rho}$. 
Thus 
\[((\theta\circ\tau)(\xi))(\gamma_i)=\gamma_i\circ(\xi^\vee)^{-1}=\pr_i\circ\hat{\rho}=\rho_i,\]
as expected.
\end{proof}

Consider $M':=(L_{\gamma_1}^M)^\times\oplus\cdots\oplus(L_{\gamma_p}^M)^\times$ and $N':=(L^N_{\rho_1})^\times\oplus\cdots\oplus(L^N_{\rho_p})^\times$ with $H$-actions described as in Lemma \ref{construction of H-bundle}. 
Let $\eta:=\eta_1\oplus\cdots\oplus\eta_p$, then $\eta: M'\ra N'$ is a holomorphic $H$-bundle isomorphism. Given any $[m_i, c_i]^M_{\gamma_i}\in L^M_{\gamma_i}$, let $[n_i, d_i]^N_{\rho_i}:=\eta_i([m_i, c_i]^M_{\gamma_i})$. Since isomorphism of line bundles preserves scaling on fibers, given any $a_i\in\Cx$, $\eta_i([m_i, c_i\cdot a_i]^M_{\gamma_i})=[n_i, d_i\cdot a_i]^N_{\rho_i}$.

The $H$-action on
$(N')^{\xi^\vee}$ becomes:
\begin{align*}
h\cdot _{{\xi^\vee}}([n_1,d_1]^N_{\rho_1},\ldots, [n_p,d_p]^N_{\rho_p})&=([n_1{\xi^\vee}(h^{-1}),d_1]^N_{\rho_1},\ldots, [n_p{\xi^\vee}(h^{-1}),d_p]^N_{\rho_p}) \\
&=([n_1,d_1\rho_1({\xi^\vee}(h^{-1}))]^N_{\rho_1},\ldots, [n_p,d_p\rho_p({\xi^\vee}(h^{-1}))]^N_{\rho_p})\\
&=([n_1,d_1\gamma_1(h^{-1})]^N_{\rho_1},\ldots, [n_p,d_p\gamma_p(h^{-1})]^N_{\rho_p}),
\end{align*}
since $\rho_i\circ\xi^\vee=\pr_i\circ\hat{\rho}\circ\hat{\rho}^{-1}\circ\hat{\gamma}=\pr_i\circ\hat{\gamma}=\gamma_i$.

On the other hand,
\begin{align*}
\eta(h\cdot  ([m_1,c_1]^M_{\gamma_1},\ldots, [m_p,c_p]^M_{\gamma_p}))&=\eta(([m_1,\gamma_1(h^{-1})c_1]^M_{\gamma_1},\ldots, [m_p,\gamma_p(h^{-1})c_p]^M_{\gamma_p}))\\
&=([n_1,d_1\gamma_1(h^{-1})]^N_{\gamma_1},\ldots, [n_p,d_p\gamma_p(h^{-1})]^N_{\gamma_p})
\end{align*}

Therefore $\eta$ is $H$-equivariant and thus a principal $H$-bundle isomorphism between $M'$ and $(N')^{\xi^\vee}$. 
Applying Claim \ref{Misom} we can conclude that $M$ and $N$ are isomorphic up to a twist by $\xi^\vee.$
\end{proof}

In particular, as a consequence of the rigidity theorem, we have:
\begin{cor}\label{rigid-isom}
Given two $(\Cx)^p$-principal bundles $M$ and $N$. If their character maps are the same, i.e. $\lambda_M=\lambda_N$, then $M$ and $N$ are isomorphic as principal $(\Cx)^p$-bundles.
\end{cor}

This tells us that the character map fully characterizes  principal $(\Cx)^p$-bundles.
In fact, we can further prove that the character map classifies principal $(\Cx)^p$-bundles.


\begin{thm}
Given $H=(\Cx)^p$ and any group homomorphism $$\lambda: \widehat{H}\ra\Pic(X),$$ then there exists a unique holomorphic principal  $H$-bundle $M$ such that $\lambda=\lambda_M$.
\end{thm}
\begin{proof}
Since $\widehat{H}=\{\chi_{k_1}\cdots\chi_{k_p}\mid k_i\in\Z\}$, let $e_i:=\chi_0\cdots\chi_1\cdots\chi_0$, where $\chi_1$ exists only on the $i$-th component, be the standard basis.
Let $M:=(\lambda(e_1))^\times\oplus\cdots\oplus(\lambda(e_p))^\times$. Let $H$ act on $M$ by
\[(h_1,\ldots,h_p)\cdot(l_1,\ldots,l_p)=(l_1h_1^{-1},\ldots,l_ph_p^{-1}).\]
Then $M$ is a holomorphic principal $H$-bundle on $X$. From Theorem \ref{discrete} we have 
\[\lambda_M(\chi_{k_1}\cdots\chi_{k_p})=k_1\lambda(e_1)+\cdots+k_p\lambda(e_p)=\lambda(\chi_{k_1}\cdots\chi_{k_p}),\] 
i.e. $\lambda_M=\lambda$.

Now suppose there exists another holomorphic principal  $H$-bundle $N$ such that $\lambda_N=\lambda.$ 
Then apply Corollary \ref{rigid-isom} we can conclude that $M$ and $N$ are isomorphic as principal $H$-bundles.
\end{proof}

\begin{cor}
Given $H=(\Cx)^p$ and any group homomorphism $$\lambda: \widehat{H}\ra\Pic(X)$$ with $K_X$ in the image of $\lambda$, then there exists a unique CY $H$-bundle $M$ such that $\lambda=\lambda_M$.
\end{cor}

\bs
Now we can prove the first part of our main theorem:
\begin{thm}
Let $X$ be a compact complex manifold. If $\Pic(X)$ is free, then $X$ admits a CY $(\Cx)^p$-bundle whose character map is an isomorphism, and the bundle is unique up to a twist by an automorphism of $(\Cx)^p$.
\end{thm}

\begin{proof}Since $X$ is compact, $\Pic(X)$ is finitely generated.
Assume $\Pic(X)\simeq \Z^p.$ Then there exist line bundles $L_1,\ldots, L_p$ such that 
\[\Pic(X)\simeq \Z L_1+\cdots+\Z L_p.\]
Let $M:=L_1^{\times}\oplus L_2^{\times}\oplus \cdots \oplus L_p^{\times}$ and $H:=(\Cx)^p$, then $M$ is a principal $H$-bundle over $X$. Then by Theorem \ref{discrete} 
\[\lambda_M(\widehat{H})= \Z L_1+\cdots+\Z L_p\simeq \Pic(X),\]
meaning that the character map is an isomorphism.

If we have another bundle $H-N\ra X$ such that 
$$\lambda_N(\widehat{H})\simeq \Pic(X),$$
then there exists $\xi_1,\ldots,\xi_p\in \widehat{H}$ such that $$\lambda_N(\xi_i)\simeq L_i, \quad \forall i.$$
Since $L_1,\ldots,L_p$ is a $\Z$-basis of  $\Pic(X)$, $\{\xi_1,\ldots,\xi_p\}$ is a $\Z$-basis of $\widehat{H}$.

Let $\text{pr}_i:H\ra\Cx$ be the $i$-th projection. Then $\{\text{pr}_1,\ldots,\text{pr}_p\}$ is a $\Z$-basis of $\widehat{H}$, and $\lambda_M(\text{pr}_i)\simeq L_i.$

Define a homomorphism $\xi: \widehat{H}\ra \widehat{H}$ where on generators $\xi(\text{pr}_i):=\xi_i$ for all $i$. Since $\xi$ maps a $\Z$-basis to a $\Z$-basis, it is an isomorphism. Moreover,
\[\lambda_N\circ\xi(\text{pr}_i)=\lambda_N(\xi_i)\simeq L_i\simeq \lambda_M(\text{pr}_i),\]
for all $i$. Thus $\lambda_N\circ\xi=\lambda_M$ and by Theorem \ref{unique} $M$ and $N$ are isomorphic up to a twist by the automorphism $\xi^\vee=(\xi_1,\ldots,\xi_p)^{-1}$.
\end{proof}


\bs

\section{Character map of the universal cover}

\subsection{Realizing $\Pic_0(X)$ for K\"ahler manifolds}\label{sect-Pic0}
In this section, we describe the connected component $\Pic_0(X)$ of the Picard group of an arbitrary connected compact K\"ahler manifold $X$. Aside from a few special cases (such as curves and tori), part of the description seems to be folklore in complex geometry, and the authors have been unable to locate a source for the general case in the literature. 
 
Let $\tilde{X}$ be the universal cover of $X$. Then $\tilde{X}$ is a principal bundle $\pi_1(X)$-bundle over $X$ and we have a character map 
$$\lambda_{\tilde{X}}:\widehat{\pi_1(X)}\ra\Pic(X),\quad\gamma\mapsto L_\gamma:=\tilde{X}\times _{\pi_1(X)} \C_{\gamma}.$$

First we describe the action of $\pi_1(X)$ on $\tilde{X}\times  \C_{\gamma}$. There is a left action of $\pi_1(X)$ on $\tilde{X}$ by deck transformation $\rho\cdot\tilde{z}$ for $\rho\in\pi_1(X)$, which we shall regard as a right action where $\tilde{z}\rho^{-1}:=\rho\cdot\tilde{z}$. 
Then $\pi_1(X)$ acts on $\tilde{X}\times  \C_{\gamma}$ by
\[\rho\cdot (\tilde{z},c):=(\tilde{z}\rho^{-1}, \gamma(\rho)c)=(\rho\cdot\tilde{z}, \gamma(\rho)c)\]
for $\rho\in\pi_1(X), \tilde{z}\in\tilde{X}, c\in\C$.

In this section we are going to prove that $\Pic_0(X)$ is contained in the image of $\lambda_{\tilde{X}}$ when $X$ is compact K\"ahler.
 

\begin{lem}{\cite[p.313]{GH}}
If $X$ is a compact K\"ahler manifold, any line bundle on $X$ with Chern class $0$ can be given by constant transition functions.
\end{lem}
Since we will need a description of the constant transition functions, we sketch a proof here.
\begin{proof}
Consider the inclusion of exact sheaf sequences on $X$:
\begin{displaymath}
\xymatrix{ 0 \ar[r] &\Z \ar[r] &\cO \ar[r]^{\exp} &\cO^\times \ar[r] & 0\\
                  0 \ar[r]  &\Z \ar[r] \ar@{=}[u] &\C \ar[r]  \ar[u]_{\iota_1} &\C^\times  \ar[r] \ar[u] _{\iota_2}&0.
}
\end{displaymath}
It induces a commutative diagram
\begin{displaymath}
    \xymatrix{ H^1(X,\mathcal{O})\ar[r]^{\alpha}                   & H^1(X,\mathcal{O^\times}) \ar[r]^{c_1}            & H^2(X, \Z)\ar@{=}[d]\\
               H^1(X,\C)         \ar[r]^{\beta} \ar[u]_{\iota^*_1} & H^1(X,\Cx)          \ar[r]\ar[u]_{\iota^*_2} & H^2(X, \Z) . }
\end{displaymath}
The map ${\iota^*_1}$ represents projection of $H^1(X, \C)\simeq H^{1,0}(X)\oplus H^{0,1}(X)=H^0(X, \Omega)\oplus H^1(X, \mathcal{O})$ onto the second factor, and so is surjective. Then 
\begin{equation}
\Pic_0(X)=\ker c_1=\im \alpha=\im \alpha\circ \iota^*_1=\im \iota^*_2\circ\beta\subseteq \im \iota^*_2.\label{eq-pic_0}
\end{equation}
It follows that any cocycle $\gamma\in H^1(X,\cO^\times)$ in the kernel of $c_1$ is in the image of $\iota^*_2$, i.e., is cohomologous to a cocycle with constant coefficients. Thus any line bundle on $X$ with Chern class $0$ can be given by constant transition functions.
\end{proof}


Equation (\ref{eq-pic_0}) says that the image of the map
\[\iota^*_2:H^1(X,\Cx)\ra H^1(X,\cO^\times)\]
contains $\Pic_0(X)$.
We want to compare it with the character map $\lambda_{\tilde{X}}$:
\begin{prop}\label{same-map} If $X$ is a complex manifold (not necessarily K\"ahler),
there exists an isomorphisms $\psi$ such that the following diagram
\[\xymatrix{\widehat{\pi_1(X)}\ar[r]^{\lambda_{\tilde{X}}}\ar[d]^{\psi}& \Pic(X)\ar@{=}[d]\\
H^1(X,\Cx)\ar[r]^{\iota^*_2} &H^1(X,\cO^\times)
}\]
commutes. 
\end{prop}

\begin{proof}
The Universal Coefficient Theorem says
\[H^1(X,\Cx)\simeq\Hom (H_1(X,\Z),\Cx)\oplus \Ext(H_0(X,\Z),\Cx).\]
Since $\Cx$ is abelian and $H_1(X,\Z)$ is the abelianization of $\pi_1(X)$,
\[\widehat{\pi_1(X)}\simeq\Hom (H_1(X,\Z),\Cx).\]
Since $H_0(X,\Z)$ is free, $\Ext(H_0(X,\Z),\Cx)=0$. Therefore we have
\[H^1(X,\Cx)\simeq\widehat{\pi_1(X)}.\]

Next, we want to describe this isomorphism $\psi$ in detail. We will use the notion of $G$-coverings which can be found in \cite[p.160]{Fu}. Let $G$ be a group. A covering $p: Y\ra X$ is called a $G$-covering if it arises from a properly discontinuous action of $G$ on $Y$. 
Since $X$ is a manifold, we can find an open cover $\cU=\{U_\alpha\}$ where $U_\alpha$ are connected and contractible. If $G$ is abelian, there are isomorphisms of abelian groups:
\begin{equation}\label{G-coverings}
\Hom(\pi_1(X), G)\simeq \{G\text{-coverings of $X$\}/isomorphism}\simeq H^1(\cU, G).
\end{equation}
The first isomorphism is given by $\rho\mapsto [Y_\rho]$ where $Y_\rho:=\tilde X\times_{\pi_1(X)} G$.
Here $\pi_1(X)$ acts on $\tilde{X}\times G$ by
$$\sigma(\tilde{z}, g)=(\tilde{z}\sigma^{-1}, \rho(\sigma)\cdot g) \text{ for }\rho\in \pi_1(X),$$
and $G$ acts on $Y_\rho$ by
$$g'\cdot[(\tilde{z}, g)]=[(\tilde{z}, g\cdot (g')^{-1})].$$
Note that our convention differs from that of \cite{Fu} in that the left and right group actions are switched. But the results there carry over.
We now specialize to $G=\Cx$. 
Let $\rho\in \widehat{\pi_1(X)}$. 
Since $U_\alpha$ is contractible, $Y_\rho|_{U_\alpha}$  is a trivial $\Cx$-covering. Then under \eqref{G-coverings}, the image of $[Y_\rho]$ in $H^1(\cU, \Cx)$ must be a collection of associated transition functions $\{g_{\alpha\beta}\}_{\alpha,\beta}$, where the $g_{\alpha\beta}:U_\alpha\cap U_\beta\ra\Cx$ are locally constant functions.

Since each $U_\alpha$ is a contractible open set, we have $H^1(U_\alpha, \Cx)=0$. Therefore, we have a canonical isomorphism $H^1(X,\Cx)\simeq H^1(\cU,\Cx)$ and we can regard $\{g_{\alpha\beta}\}_{\alpha,\beta}$ as an element in $H^1(X,\Cx)$.
Define $$\psi: \widehat{\pi_1(X)}\ra H^1(X,\Cx), \rho\mapsto \{g_{\alpha\beta}\}_{\alpha,\beta},$$
then $\psi$ is an isomorphism.

Moreover, since $\iota_2$ is an inclusion, $\iota^*_2(\{g_{\alpha\beta}\}_{\alpha,\beta})=\{\iota_2(g_{\alpha\beta})\}_{\alpha,\beta}=\{g_{\alpha\beta}\}_{\alpha,\beta}$. Therefore \[\iota^*_2\circ\psi(\rho)=\{g_{\alpha\beta}\}_{\alpha,\beta}.\]

\bs

Now we consider 
$$\lambda_{\tilde{X}}(\rho)=L_\rho:=\tilde{X}\times _{\pi_1(X)} \C_{\rho}$$
where $[\tilde{z},c]=[\sigma\cdot\tilde{z}, \rho(\sigma)c]$. 
$\cU$ defined above is a local trivialization of $L_\rho$, and from the way we define $Y_\rho$ and $L_\rho$ we can see that they have the same transition functions. Therefore 
\[\lambda_{\tilde{X}}(\rho)=\{g_{\alpha\beta}\}_{\alpha,\beta}=\iota^*_2\circ\psi(\rho),\]
which means that the diagram in the proposition commutes, as desired.
\end{proof}

An immediate result of this proposition is the following:
\begin{cor}\label{same-map2} 
If $X$ is a compact K\"ahler manifold, $\Pic_0(X)$ is contained in the image of $\lambda_{\tilde{X}}$. That is, given a line bundle $L\in\Pic_0(X)$, there exists a character $\gamma\in \widehat{\pi_1(X)}$ such that $L\simeq \tilde{X}\times _{\pi_1(X)} \C_{\gamma}$.
\end{cor}
{
\begin{rem}
There is another way to prove Corollary \ref{same-map2} by giving an explicit description of the character $ \gamma$ by {\it factors of automorphy} (cf. \cite[Appendix B]{LB}) under the assumption that $X$ is  compact K\"ahler. It  is a generalization of the case of complex tori in \cite[p.313,308]{GH}.
\end{rem}
}

\begin{exm}\label{exm-torus}
Let $X$ be a complex torus of dimension $g$. Then $\pi_1(X)\simeq \Z^{2g}$.
Corollary \ref{same-map2} tells us that the image of the character map
\[\lambda_{\tilde{X}}: \widehat{\pi_1(X)}\simeq (\Cx)^{2g}\ra \Pic(X)\]
contains $\Pic_0(X)\simeq \C^g/\Z^{2g}$. 
And it is clear from a dimension argument that this map has a kernel. We will describe the kernel in Subsection \ref{sec-kernel}.
\end{exm}

\subsection{More about the character map $\lambda_{\tilde{X}}$}
In this section we are going to discuss two important facts about $\lambda_{\tilde{X}}$.

\bs
Consider the following commutative diagram induced by inclusion of  the exponential sheaf sequences:
\begin{equation}\label{exact sequences}
\xymatrix{H^1(X,\cO)\ar[r]^\alpha &H^1(X,\cO^\times)\ar[r]^{c_1} &H^2(X,\Z)\ar[r] &H^2(X,\cO)\\
H^1(X,\C)\ar[r]^\beta\ar[u]_{\iota^*_1} &H^1(X,\Cx)\ar[r]^\gamma\ar[u]_{\iota^*_2} &H^2(X,\Z)\ar[r]^\epsilon\ar@{=}[u] &H^2(X,\C).\ar[u]}
\end{equation}
Since $X$ is compact, $H_i(X):=H_i(X,\Z)$ is a finitely generated abelian group for all $i$.
We denote:
\[q:=\text{rank}(H_1(X)),\; T_1:= H_1(X)_{\text{tor}},\; b:=\text{rank}(H_2(X)),\;  T_2:= H_2(X)_{\text{tor}}\]
where $q,b\in\Z $ and $T_1,T_2$ are finite groups.
Then we have isomorphisms:
\[H_1(X)\simeq \Z^q\oplus T_1,\quad H_2(X)\simeq \Z^b\oplus T_2.\]

The Universal Coefficient Theorem tells us that
\begin{align*}
&H^1(X,\C)\simeq\Hom(H_1(X),\C)\oplus  \Ext(H_0(X),\C)\simeq \C^q;\\
&H^1(X,\Cx)\simeq \Hom(H_1(X),\Cx)\oplus \Ext(H_0(X),\Cx)\simeq  (\Cx)^q\oplus T_1;\\
&H^2(X,\Z)\simeq \Hom(H_2(X),\Z)\oplus\Ext(H_1(X),\Z)\simeq \Z^b\oplus T_1;\\
&H^2(X,\C\simeq\Hom(H_2(X),\C)\oplus\Ext(H_1(X),\C)\simeq \C^b.
\end{align*}

So we can rewrite the bottom exact sequence in \eqref{exact sequences} as:
\[\C^q\xrightarrow{\beta}(\Cx)^q\oplus T_1\xra{\gamma} \Z^b\oplus T_1\xra{\epsilon}\C^b.\]

We first look at $\epsilon$. Since $\C$ is free, $\Hom(T_1,\C^b)=0$. And since the functor $\Hom(H_2(X),-)$ is left exact, $\epsilon$ restricted to $\Z^b$ is an inclusion. Thus $\im \gamma=\ker \epsilon\simeq T_1$.

We claim that $\Hom(\Cx, T_1)=0.$ Let $N$ be an integer such that $Nt=0$ for all $t\in T_1.$ Since $\Cx$ is divisible, for any $a\in\Cx$, there exists some 
$b\in\Cx$ such that $a=b^N.$
Given any $f\in \Hom(\Cx, T_1)$, $f(a)=f(b^N)=Nf(b)=0.$ Thus $\Hom(\Cx, T_1)=0.$
So $\gamma$ restricted to $(\Cx)^q$ maps to $0$.
Therefore $\gamma$ restricted to $T_1$ is an isomorphism to the torsion part of $H^2(X,\Z).$

Proposition \ref{same-map} tells us that 
$\lambda_{\tilde{X}}=\iota^*_2\circ\psi$
and $\psi$ is an isomorphism.
Thus 
$$\im c_1\circ\lambda_{\tilde{X}}=\im c_1\circ \iota^*_2 =\im \gamma=H^2(X,\Z)_{\text{tor}}.$$ 

Also natural inclusions $\im c_1\circ \iota^*_2 \subset \im c_1\subset H^2(X,\Z)$ imply that $$H^2(X,\Z)_{\text{tor}}=(\im c_1\circ \iota^*_2 )_{\text{tor}}\subset(\im c_1)_{\text{tor}}\subset H^2(X,\Z)_{\text{tor}}.$$
It forces $c_1(\Pic(X))_{\text{tor}}= H^2(X,\Z)_{\text{tor}}$ and $c_1\circ\lambda_{\tilde{X}}$ maps onto the torsion of $c_1(\Pic(X))$. 

To be more precise, we look at the character group $\Hom(\pi_1(X),\Cx)$:
\begin{align*}
\Hom(\pi_1(X),\Cx)&\simeq \Hom(\overline{\pi_1(X)},\Cx)\simeq \Hom(\Z^q\oplus T_1,\Cx)\\
&\simeq \Hom(\Z^q,\Cx)\oplus \Hom(T_1,\Cx)\simeq (\Cx)^q\oplus T_1
\end{align*}
where $\overline{\pi_1(X)}$ denotes the abelianization of $\pi_1(X)$.
Then we can see further that $c_1\circ\lambda_{\tilde{X}}$ restricted on $\Hom(T_1,\Cx)$ gives an isomorphism to $c_1(\Pic(X))_{\text{tor}}$.

To sum up, we have:
\begin{prop}\label{torsion}
If $X$ is a compact complex manifold (not necessarily K\"ahler), then $c_1\circ \lambda_{\tilde{X}}$ maps onto the torsion of $c_1(\Pic(X))$. And $c_1\circ\lambda_{\tilde{X}}$ restricted on $\Hom(T_1,\Cx)$ gives an isomorphism to $c_1(\Pic(X))_{\text{tor}}$.
\end{prop}

From the above discussion we can also conclude that $\im \beta=\ker \gamma=(\Cx)^q.$

If we assume that $X$ is K\"ahler, we have
\[\iota^*_2((\Cx)^q)\simeq \im \iota^*_2\circ \beta=\im \alpha\circ \iota^*_1=\Pic_0(X),\]
where the last equality comes from (\ref{eq-pic_0}).

Using the isomorphism $\psi$ between $\widehat{\pi_1(X)}$ and $H^1(X,\Cx)$ again we can conclude that 
\[\lambda_{\tilde{X}}(\Hom(\Z^q,\Cx))=\Pic_0(X).\]

To sum up, we have:
\begin{prop}\label{free-to-pic_0}
Let $X$ be a compact K\"ahler manifold, 
then the restriction of the character map $\lambda_{\tilde{X}}$ to $\Hom(\Z^q,\Cx)$ maps exactly onto $\Pic_0(X).$
\end{prop}

\subsection{Further description of $\Pic_0(X)$ and kernel of $\lambda_{\tilde{X}}$}\label{sec-kernel}
The character map $\lambda_{\tilde{X}}$ is usually not injective.

Let $X$ be a compact K\"ahler manifold. First we want to give a more detailed description of $\Pic_0(X)$.
Consider the following commutative diagram:
\[\xymatrix{ 0\ar[r] &\Z\ar[r]\ar@{=}[d]&\R\ar[r]^{\exp}\ar@{^(->}[d]_{i_1}&S^1\ar@{^(->}[d]_{i_2}\ar[r]&0\\
                   0\ar[r] &\Z\ar[r]\ar@{=}[d]&\C\ar[r]^{\exp} \ar@{^(->}[d]_{\iota_1}&\Cx\ar[r]\ar@{^(->}[d]_{\iota_2}&0\\
                 0\ar[r] &\Z\ar[r] &\cO \ar[r]^{\exp} & \cO^\times\ar[r]&0
}\]
where $S^1$ denotes the unit circle. The horizontal lines are exact (but vertical lines are not.)

It induces a diagram of long exact sequences:
\[\xymatrix{ H^0(X,S^1)\ar[r]^{\delta_1}\ar[d]& H^1(X,\Z)\ar[r]\ar@{=}[d] & H^1(X,\R)\ar[r]^{\exp^*_1}\ar[d]_{i^*_1} &H^1(X,S^1)\ar[r]\ar[d]_{i^*_2} &H^2(X,\Z)\ar@{=}[d] \\
H^0(X,\Cx)\ar[r]^{\delta_2}\ar[d]& H^1(X,\Z)\ar[r]\ar@{=}[d] & H^1(X,\C)\ar[d]_{\iota^*_1} \ar[r]^{\exp^*_2}&H^1(X,\Cx)\ar[r]\ar[d]_{\iota^*_2} &H^2(X,\Z)\ar@{=}[d] \\
H^0(X,\cO^\times)\ar[r]^{\delta_3}& H^1(X,\Z)\ar[r] & H^1(X,\cO)\ar[r]^{\exp^*_3} &H^1(X,\cO^\times)\ar[r]^{c_1} &H^2(X,\Z) \\
}\]
where the rows are exact but not necessarily the columns.
By universal coefficient theorem and that $\Hom(\Z,-)$ is exact, we see that $\delta_1,\delta_2,\delta_3 $ are inclusions.

From the universal coefficient theorem it is also clear that $$\im (\exp^*_1)\simeq \Hom(\Z^q, S^1), \quad \im (\exp^*_2)\simeq \Hom(\Z^q, \Cx)$$
where $\Z^q\simeq H_1(X)_{\text{free}}$. Since $\im(\exp^*_3)=\Pic_0(X)$, we can rewrite the middle part of the sequences as

\[\xymatrix{ 0\ar[r]& H^1(X,\Z)\ar[r]\ar@{=}[d] & H^1(X,\R)\ar[r]^{\exp^*_1\quad}\ar[d]_{i^*_1} &\Hom(\Z^q, S^1)\ar[r]\ar[d]_{i^*_2} &0 \\
0\ar[r]& H^1(X,\Z)\ar[r]\ar@{=}[d] & H^1(X,\C)\ar[d]_{\iota^*_1} \ar[r]^{\exp^*_2\quad}&\Hom(\Z^q, \Cx)\ar[r]\ar[d]_{\lambda_{\tilde{X}}} &0 \\
0\ar[r]& H^1(X,\Z)\ar[r] & H^1(X,\cO)\ar[r]^{\exp^*_3} &\Pic_0(X)\ar[r] &0 \\
}\]

\begin{claim}If $X$ is compact K\"ahler,
$\iota^*_1\circ i^*_1$ is an isomorphism between $H^1(X,\R)$ and $H^1(X,\cO)$ as real vector spaces.
\end{claim}

\begin{proof}
Since we assume $X$ to be K\"ahler, $H^1(X,\C)=H^1(X,\cO)\oplus\overline{H^1(X,\cO)}$. For
$H^1(X,\R)\subset H^1(X,\C)$ is real, $$H^1(X,\R)\cap H^1(X,\cO)=H^1(X,\R)\cap \overline{H^1(X,\cO)}=0.$$

Since $\iota^*_1$ is the projection from the Hodge decomposition, $\ker(\iota^*_1)=\overline{H^1(X,\cO)}$. It follows that $\ker(\iota^*_1)\cap H^1(X,\R)=0$. Since $i^*_1$ is an inclusion, $\ker(\iota^*_1\circ i^*_1)\cap  H^1(X,\R)=0$. Thus $\iota^*_1\circ i^*_1$ is injective.
Since $H^1(X,\R)$ and $H^1(X,\cO)$ have the same dimension as real vector spaces, $\iota^*_1\circ i^*_1$ is an isomorphism.
\end{proof}

From the universal coefficient theorem it is clear that $i^*_2$ is the natural inclusion. Now the first and third rows become
\[\xymatrix{ 0\ar[r]& H^1(X,\Z)\ar[r]\ar@{=}[d] & H^1(X,\R)\ar[r]^{\exp^*_1\quad}\ar[d]_{\iota^*_1\circ i^*_1} &\Hom(\Z^q, S^1)\ar[r]\ar[d]_{\lambda_{\tilde{X}}\circ i^*_2} &0 \\
0\ar[r]& H^1(X,\Z)\ar[r] & H^1(X,\cO)\ar[r]^{\exp^*_3} &\Pic_0(X)\ar[r] &0.\\
}\]
Finally, the five-lemma yields:
\begin{prop}\label{thm pic0} If $X$ is compact K\"ahler,
$\lambda_{\tilde{X}}\circ i^*_2$ restricted to $\Hom(\Z^q, S^1)\subset H^1(X,S^1)$ gives an isomorphism to $\Pic_0(X)$ as real tori.
\end{prop}

\bs

Now we want to describe the kernel of $\lambda_{\tilde{X}}$. 
Proposition \ref{same-map} tells us that $\ker(\lambda_{\tilde{X}})\simeq\ker (\iota^*_2)$. We consider the following exact sheaf sequences\footnote{We thank Prof. Alan Mayer for pointing out this diagram to us.}:
\[\xymatrix{ &0\ar[d]&0\ar[d]\\
                   0\ar[r] &\Z\ar[r]\ar[d]_i&\Z\ar[r]\ar[d]_i&0\ar[d]\\
                   0\ar[r] &\C\ar[r]^{\iota_1} \ar[d]_{\exp}&\cO \ar[r]^d \ar[d]_{\exp}&\Omega^1_c\ar[r]\ar[d]&0\\
                 0\ar[r] &\Cx\ar[d]\ar[r]^{\iota_2} &\cO^\times \ar[d]\ar[r]^{d\log} & \Omega^1_c\ar[r]\ar[d]&0\\
                 &0&0&0
}\]
where $\Omega^1_c$ is the sheaf of holomorphic closed $1$-forms on $X$.
Then we have a long exact sequence:
\[0\ra H^0(X, \Cx)\ra H^0(X, \cO^\times )\ra H^0(X, \Omega^1_c)\xrightarrow{\delta} H^1(X, \Cx)\xrightarrow{\iota^*_2} H^1(X, \cO^\times)\ra\cdots\]
Since $X$ is compact, $H^0(X, \cO^\times )\simeq\Cx\simeq H^0(X, \Cx)$. Thus $\delta$ is an injection. Thus $\ker (\iota^*_2)=\delta(H^0(X, \Omega^1_c))\simeq H^0(X, \Omega^1_c)$.
Therefore we have:
\begin{prop}
If $X$ is a compact K\"ahler manifold, then $\ker(\lambda_{\tilde{X}})\simeq H^0(X, \Omega^1_c)$.
\end{prop}



Proposition \ref{torsion} and Proposition \ref{free-to-pic_0} tells us that there is a short exact sequence
\[0\ra \ker \lambda_{\tilde{X}}\ra \Hom (\Z^q,\Cx)\xrightarrow{\lambda_{\tilde{X}}}  \Pic_0(X)\ra 0.\]
Proposition \ref{thm pic0} tells us that $\lambda_{\tilde{X}}$ restricted to $\Hom (\Z^q, S^1)$ is an isomorphism to $\Pic_0(X)$. 
Therefore the inverse of $\lambda_{\tilde{X}}\circ i^*_2$ splits this sequence of groups, and  
so we have
\[\Hom(\Z^q,\Cx)\simeq  \ker \lambda_{\tilde{X}}\oplus \Pic_0(X)\]
(Note that this splitting is not holomorphic.) 

\bs

\begin{exm}
Let $X$ be a compact Riemann surface with genus $g$. Then every homomorphic $1$-form on $X$ is closed, i.e. $\Omega^1_c=\Omega^1$. We know that $H^0(X, \Omega^1)\simeq \C^g$, thus in this case $\Hom(\pi_1(X),\Cx)\simeq (\Cx)^{2g}$, $\Pic_0(X)\simeq \C^g/\Z^{2g}$ and $\ker(\lambda_{\tilde{X}})\simeq\C^g.$
\end{exm}

\bs

\section{Existence of CY bundles on K\"ahler manifolds}\label{sect-Pic}
Let $X$ be a connected compact complex manifold, then $H^2(X, \Z)$ is a finitely generated abelian group. 
As a subgroup of $H^2(X,\Z)$, $c_1(\Pic(X))$ is a finitely generated abelian group as well.

Let $p$ be the rank of $c_1(\Pic(X))$ and $L_1,\cdots,L_p$ be holomorphic line bundles on $X$ such that $\{c_1(L_1), \cdots, c_1(L_p)\}$ generate the free part of $c_1(\Pic(X))$. Let $M:=L_1^{\times}\oplus L_2^{\times}\oplus \cdots \oplus L_p^{\times}$, then $M$ is a principal $(\Cx)^p$-bundle over $X$. Let
$\tilde{M}:=\tilde{X}\oplus M$ be the Whitney sum of the universal cover $\tilde X$ and $M$  over $X$. 
Let $\tilde{H}:=\pi_1(X)\times(\Cx)^p$. Then $\tilde{M}$ is a principal $\tilde{H}$-bundle over $X$.


\begin{prop}Let $X$ be  compact K\"ahler,
then the character map
\[\lambda_{\tilde{M}}:\widehat{\tilde{H}}\ra\Pic(X),\quad \chi\mapsto \tilde{M}\times_{\tilde{H}}\C_{\chi}\]
is surjective.
\end{prop}

\begin{proof}
We are going to show that given any line bundle $L$ on $X$, there exists an element $\chi$ in $\widehat{\tilde{H}}$ such that $\tilde{M}\times_{\tilde{H}}\C_{\chi}=L$.

First note that
\[\widehat{\tilde{H}}\simeq\widehat{\pi_1(X)}\times \widehat{(\Cx)^p}\]
By assumption on $L_1,\ldots,L_p$, there exist integers $k_1,\ldots,k_p$ such that $$\tilde{\sigma}:=c_1(L)-(k_1c_1(L_1)+\cdots+k_pc_1(L_p))\in c_1(\Pic(X))_{\text{tor}}.$$
By Proposition \ref{torsion} there exists $\rho\in \widehat{\pi_1(X)}$ such that $$c_1( \tilde{X}\times_{\pi_1(X)}\C_{\rho})=\tilde{\sigma},$$
i.e.,
\[c_1( \tilde{X}\times_{\pi_1(X)}\C_{\rho})=c_1(L-(k_1L_1+\cdots+k_pL_p))\]

 Then
$$L_0:=L-(k_1L_1+\cdots+k_pL_p)- \tilde{X}\times_{\pi_1(X)}\C_{\rho}\in\Pic_0(X).$$

By Corollary \ref{same-map2} there exists $\gamma\in \widehat{ \pi_1(X)}$ such that
$$L_0\simeq \tilde{X}\times_{\pi_1(X)}\C_{\gamma}.$$
Then we have
\[\tilde{M}\times_{\tilde{H}}\C_{(\gamma\cdot \rho,  \chi_0\cdots \chi_0 )}\simeq L_0+\tilde{X}\times_{\pi_1(X)}\C_{\rho}\]
where $\chi_0\cdots\chi_0$ is the trivial character in $\widehat{(\Cx)^p}$.

Theorem \ref{discrete} tells us that
$$M\times_{(\Cx)^p}\C_{\chi_{k_1}\cdots\chi_{k_p}}\simeq k_1L_1+\cdots+k_pL_p,$$
then
\[\tilde{M}\times_{\tilde{H}}\C_{(e,\chi_{k_1}\cdots\chi_{k_p})}\simeq k_1L_1+\cdots+k_pL_p\]
where $e$ denotes the trivial character $e:\pi_1(X)\ra\Cx, \rho\ra 1, \forall \rho\in \pi_1(X)$.

Since
\[\widehat{\tilde{H}}\ra Pic(X), \quad \chi\mapsto \tilde{M}\times_{\tilde{H}} \C_{\chi}\]
is a group homomorphism, and
\[(e,\chi_{k_1}\cdots\chi_{k_p})\cdot(\gamma\cdot \rho, \chi_0\cdots \chi_0)=(\gamma\cdot \rho, \chi_{k_1}\cdots\chi_{k_p}),\]
we have
\[\tilde{M}\times_{\tilde{H}}\C_{(\gamma\cdot \rho,\chi_{k_1}\cdots\chi_{k_p})}\simeq L_0+\tilde{X}\times_{\pi_1(X)}\C_{\rho}+k_1L_1+\cdots+k_pL_p=L.\]
I.e., $\tilde{M}\times_{\tilde{H}}\C_\chi\simeq L$ where $\chi=(\gamma\cdot \rho,\chi_{k_1}\cdots\chi_{k_p})$.
\end{proof}

In particular, since $K_X\in \Pic(X)$, we conclude that:
\begin{thm}\label{kahler onto}
If $X$ is compact K\"ahler, then $\tilde{H}-\tilde{M}\ra X$ constructed as above is a CY bundle whose character map is onto. And $\ker\lambda_{\tilde{M}}\simeq H^0(X, \Omega^1_c)$.
\end{thm}

\begin{proof}
The first statement is clear. 

For the second statement, we have $\Hom(\tilde{H},\Cx)=\widehat{\pi_1(X)}\times\widehat{(\Cx)^p}$. Consider the character map $\lambda_{\tilde{M}}$ restricted to $\widehat{(\Cx)^p}$:
\[\Z^p\simeq\widehat{(\Cx)^p}\ra \Pic(X), ~~\chi:=(e,\chi_{k_1}\cdots\chi_{k_p})\mapsto \tilde{M}\times_{\tilde{H}} \C_\chi=k_1L_1+\cdots+k_pL_p.\]
From the choice of $L_1,\ldots,L_p$ we know that $c_1(L_1),\ldots,c_1(L_p)$ are $\Z$-independent, which implies that $L_1,\ldots,L_p$ are also $\Z$-independent, meaning that the above map is an injection. Thus  we can conclude that \[\ker\lambda_{\tilde{M}}\simeq \ker \lambda_{\tilde{X}}\simeq\ker \iota^*_2.\]
\end{proof}



In particular, we can further conclude that 
\begin{prop}
If $X$ is compact K\"ahler and $\Pic(X)$ is finitely generated, then $\tilde{H}-\tilde{M}\ra X$ is a CY bundle whose character map is an isomorphism.
\end{prop}
\begin{proof}
When $\Pic(X)$ is finitely generated, $H^0(X,\Omega^1)=0.$ If $\Omega^1_c$ has a nontrivial global section, it will be  nontrivial in $H^0(X,\Omega^1)$ as well. Therefore $H^0(X, \Omega^1_c)=0$ and by Corollary \ref{kahler onto}, $\lambda_{\tilde{M}}$ is an isomorphism.
\end{proof}

\begin{rem}
Actually this statement holds without the assumption that $X$ is K\"ahler. It will be discussed in the next section.
\end{rem}

\bs

\section{Existence of CY bundles when $\Pic(X)$ is finitely generated}\label{sect-discrete}

We are going to study compact complex manifolds, not necessarily K\"ahler, with the assumption that their Picard groups are finitely generated.
We want to make use of our previous constructions,   
and the idea is to construct new principal bundles out of the old one.

\subsection{Principal bundles}

Given a holomorphic Lie group homomorphism $f: H\ra K$, it induces a homomorphism $\widehat{f}: \widehat{K}\ra\widehat{H},~~\chi\mapsto \chi\circ f$. The map $f$ also induces a holomorphic action of $H$ on $K$: $H\times K\ra K, ~~ (h,k)\mapsto f(h)\cdot k$.

If we are given a principal $H$-bundle $M$ over $X$, let $N:=M\times_H K$ where the action of $H$ on $K$ is induced by $f$. Then $N$ is a holomorphic $K$-bundle over $X$. Define a $K$-action on $N$ by right translation on fibers: $$K\times N\ra N, ~~(k,[m,l])\mapsto [m,l\cdot k^{-1}].$$
It is clear that this action preserves the fibers and acts freely and transitively on them. Thus with this action $N$ is a principal $K$-bundle over $X$, we call it {\it the principal $K$-bundle induced by $f$ and $M$}.

\begin{prop}\label{new bundle}
If $N$ is the principal $K$-bundle induced by $f$ and $M$, then the diagram of character maps
\[\xymatrix{\widehat{K}\ar[r]^{\lambda_N\quad}\ar[d]_{\widehat{f}}&\Pic(X)\\
\widehat{H}\ar[ru]_{\lambda_M}
}\]
commutes.
\end{prop} 
\begin{proof}
It suffices to show that given a character $\chi\in \widehat{K}$, there is a holomorphic line bundle isomorphism $N\times_K\C_\chi\simeq M\times_H\C_{\chi\circ f}.$
Define a map
\begin{align*}
\xi: (M\times_H K)\times_K \C_\chi&\ra M\times_H\C_{\chi\circ f}\\
[[m,l],c]&\mapsto [m,\chi(l)c]
\end{align*}
Since 
\begin{align*}
\xi([[mh^{-1},f(h)l],c])&=[mh^{-1}, \chi(f(h)l)c]=[mh^{-1}, \chi(f(h))\chi(l)c]=[m,\chi(l)c],\\
\xi([[m,lk^{-1}],\chi(k)c])&=[m,\chi(lk^{-1})\chi(k)c]=[m,\chi(l)c],
\end{align*}
$\xi$ is well defined. It is clear that $\xi$ induces a linear isomorphism on fibers and commutes with quotients to $X$, thus it is an isomorphism between holomorphic line bundles.
\end{proof}

\subsection{Principal $T_1$-bundles}
The abelianization of $\pi_1(X)$ gives a Lie group homomorphism $\text{Ab}: \pi_1(X)\ra\overline{\pi_1(X)}$.
Let $\bar{X}:=\tilde{X}\times_{\pi_1(X)} \overline{\pi_1(X)}$ be the principal
$\overline{\pi_1(X)}$-bundle induced by $\text{Ab}$ and $\tilde{X}$.
Then $\bar{X}$ is a cover of $X$ whose deck transformation group is $\overline{\pi_1(X)}$.
Proposition \ref{new bundle} implies the following diagram
\begin{equation}\label{abel-pi-1}
\xymatrix{\widehat{\overline{\pi_1(X)}}\ar[r]^{\lambda_{\bar{X}}}\ar[d]_{\widehat{\text{Ab}}}&\Pic(X)\\
            \widehat{\pi_1(X)}\ar[ru]_{\lambda_{\tilde{X}}}
}
\end{equation}
commutes.
Since $\Cx$ is abelian, $\widehat{\text{Ab}}$ is an isomorphism.

Since $\overline{\pi_1(X)}\simeq H_1(X)\simeq\Z^q\oplus T_1$, we choose an isomorphism $\phi:\overline{\pi_1(X)}\xrightarrow{\simeq} \Z^q\oplus T_1$. Let $\pr_2: \Z^q\oplus T_1\ra T_1$ denote the projection to the second factor.

 Then  we have a homomorphism $\pr_2\circ\phi: \overline{\pi_1(X)}\ra T_1$.
Let $T:=\bar{X}\times_{\overline{\pi_1(X)}} T_1$ be the principal $T_1$-bundle induced by $\pr_2\circ\phi$ and $\bar{X}$.
Then it follows from Proposition \ref{new bundle} that there is a commutative diagram 
\[\xymatrix{\widehat{T_1}\ar[r]^{\lambda_T}\ar[d]_{\widehat{\text{pr}_2\circ\phi}}&\Pic(X)\\
\widehat{\overline{\pi_1(X)}}.\ar[ru]_{\lambda_{\bar{X}}}
}\]
It is clear that $\widehat{\pr_2\circ\phi}$ gives an isomorphism from $\widehat{T_1}$ to $\Hom(T_1,\Cx)\subset \widehat{\overline{\pi_1(X)}}.$

Then by Proposition \ref{torsion}, $c_1\circ\lambda_T=c_1\circ (\lambda_{\bar{X}}\circ (\widehat{\pr_2\circ\phi}))$ gives an isomorphism from $\widehat{T_1}$ to $c_1(\Pic(X))_{\text{tor}}$.

If we assume that $\Pic(X)$ is finitely generated, then $\Pic_0(X)=0$ and $c_1$ is an injection. It further implies that $\lambda_T$ is an isomorphism from $\widehat{T_1}$ to $\Pic(X)_{\text{tor}}$.

\subsection{General case}
Assume that $\Pic(X)$ is finitely generated. Let $p$ be the rank of $\Pic(X)$ and $L_1,\cdots,L_p$ be holomorphic line bundles that generate  $\Pic(X)_{\text{free}}$.
 Let $M:=L_1^{\times}\oplus L_2^{\times}\oplus \cdots \oplus L_p^{\times}\oplus T$, then $M$ is a principal $((\Cx)^p\oplus T_1)$-bundle over $X$.
 
It is clear that $\lambda_M$ maps onto $\Pic(X)$ in this case. And since $$\Hom((\Cx)^p\oplus T_1,\Cx)\simeq \Z^p\oplus T_1\simeq \Pic(X),$$
$\lambda_M$ is an isomorphism.
\begin{thm}
Given a compact complex manifold $X$ (not necessarily K\"ahler),  if $\Pic(X)$ is finitely generated, then
there exists a CY bundle over $X$ whose character map is an isomorphism.
\end{thm}

\bs
\section{Existence of CY bundles for abelian structure groups}\label{sect-abelian}
In this section we are going to study CY bundles whose structure group is an abelian Lie group.


\subsection{Connected abelian complex Lie groups}
\begin{defn}
A toroidal group is an abelian complex Lie group on which every holomorphic function is constant.
\end{defn}
In particular, the character group of a toroidal group is trivial.

\bs
There is a very nice result on the classification of connected abelian complex groups by Remmert and by Morimoto {(cf. \cite{Mo} and \cite[p.6]{AK})}:
\begin{thm}[Decomposition of abelian Lie groups]
Every connected abelian complex Lie group is holomorphically isomorphic to a group of the form
\[(\Cx)^a\times\C^b\times G_0,\]
where $G_0$ is a toroidal group. The decomposition is unique.
\end{thm}

\subsection{Principal $\C^q$-bundles}
In our construction of CY bundles in Section \ref{sect-Pic},  
$\tilde{H}=\pi_1(X)\times(\Cx)^a$ is neither abelian nor connected in general. So we need to make changes to the previous construction. 

First we observe that:
\begin{prop}\label{abelianpi1}If $X$ is a compact K\"ahler manifold,
then the image of the character map of $\bar{X}$ 
contains $\Pic_0(X)$. Moreover, 
$\lambda_{\bar{X}}(\Hom(\Z^q,\Cx))=\Pic_0(X).$
\end{prop}
\begin{proof}
Look at diagram \eqref{abel-pi-1}:
\[\xymatrix{\widehat{\overline{\pi_1(X)}}\ar[r]^{\lambda_{\bar{X}}}\ar[d]_{\widehat{\text{Ab}}}^\simeq&\Pic(X)\\
            \widehat{\pi_1(X)}.\ar[ru]_{\lambda_{\tilde{X}}}
}\]
From Corollary \ref{same-map2} we know that $\Pic_0(X)$ is contained in the image of $\lambda_{\tilde{X}}$, thus the image of $\lambda_{\bar{X}}$ contains $\Pic_0(X)$ as well.

Proposition \ref{free-to-pic_0} shows further that $\lambda_{\tilde{X}}(\Hom(\Z^q,\Cx))=\Pic_0(X)$, composing with the isomorphism $\widehat{\text{Ab}}$ we have 
$\lambda_{\bar{X}}(\Hom(\Z^q,\Cx))=\Pic_0(X).$
\end{proof}



Let $\pr_1: \overline{\pi_1(X)}\simeq\Z^q\oplus T\ra\Z^q$ be the projection to the first factor.
 Then  we have a homomorphism $\pr_1\circ\phi: \overline{\pi_1(X)}\ra \Z^q$.
Let $N:=\bar{X}\times_{\overline{\pi_1(X)}} \Z^q$ be the principal $\Z^q$-bundle induced by $\pr_1\circ\phi$ and $\bar{X}$.
Then Proposition \ref{new bundle} gives a commutative diagram 
\[\xymatrix{\widehat{\Z^q}\ar[r]^{\lambda_N}\ar[d]_{\widehat{\text{pr}_1\circ\phi}}&\Pic(X)\\
\widehat{\overline{\pi_1(X)}}.\ar[ru]_{\lambda_{\bar{X}}}
}\]
It is clear that $\widehat{\pr_1\circ\phi}$ induces an isomorphism from $\widehat{\Z^q}$ to $\Hom(\Z^q,\Cx)\subset  \widehat{\overline{\pi_1(X)}}$. Then by Proposition \ref{abelianpi1} $\Pic_0(X)$ is contained in the image of the composition map $\lambda_N$.

Next we consider the inclusion $\iota: \Z^q\ra\C^q$. Let $Q:=N\times_{\Z^q}\C^q$ be the principal $\C^q$-bundle induced by $\iota$ and $N$. Then from Proposition \ref{new bundle} we have a commutative diagram
\[\xymatrix{\widehat{\C^q}\ar[r]^{\lambda_Q\quad}\ar[d]_{\widehat\iota}&\Pic(X)\\
\widehat{\Z^q}.\ar[ru]_{\lambda_N}
}\]
Since $\Cx$ is a divisible group, $\Hom(-,\Cx)$ is an exact contravariant functor. 
So $\widehat{ \iota}$ is surjective and thus $\Pic_0(X)$ is contained in the image of the character map $\lambda_Q$.

To sum up, we have:
\begin{prop}\label{Cq-exist}
Let $X$ be a compact K\"ahler manifold and $q$ be the rank of $\overline{\pi_1(X)}$. Then there exists a principal $\C^q$-bundle over $X$ such that the image of its character map contains $\Pic_0(X)$.
\end{prop}

\subsection{General case} 
We now assemble results of the preceding sections to construct CY bundles over $X$ with arbitrary abelian structure groups.

Since $\Pic(X)/\Pic_0(X)\simeq c_1(\Pic(X))\subset H^2(X)$, it is finitely generated. Let ${p'}\in\Z$ be the minimum number of generators of the whole group $c_1(\Pic(X))$ and let $\{c_1(L_1),\ldots,c_1(L_{p'})\}$ be a set of generators. 
Let $q$ be the rank of $\overline{\pi_1(X)}$ and $Q$ be the principal $\C^q$-bundle over $X$ such that  $\Pic_0(X)\subset \im\lambda_Q$, as described in the previous subsection. Let $P:=L_1^{\times}\oplus L_2^{\times}\oplus \cdots \oplus L_{p'}^{\times}\oplus Q$, 
then it is a principal $((\Cx)^{p'}\times\C^q)$-bundle over $X$.

\begin{prop} If $X$ is compact K\"ahler,  then $P$ defined as above is a CY $((\Cx)^{p'}\times\C^q)$-bundle whose character map is onto.
\end{prop}

\begin{proof}

First, we can identify
$\widehat{(\Cx)^{p'}\times\C^q}\equiv\widehat{(\Cx)^{p'}}\times \widehat{\C^q}.$

By assumption on $L_1,\ldots,L_{p'}$, there exist integers $k_1,\ldots,k_{p'}$ such that $c_1(L)=k_1c_1(L_1)+\cdots+k_{p'}c_1(L_{p'})$. Thus
$$L_0:=L-(k_1L_1+\cdots+k_{p'}L_{p'})\in\Pic_0(X).$$

By Proposition \ref{Cq-exist}, there exists $\gamma\in \widehat{\C^q}$ such that
$$L_0\simeq Q\times_{\C^q}\C_{\gamma}.$$
Hence
\[P\times_{(\Cx)^{p'}\times\C^q}\C_{(\chi_{k_1}\cdots\chi_{k_{p'}},\gamma)}\simeq k_1L_1+\cdots+k_{p'}L_{p'}+L_0=L.\]
I.e., $\lambda_P(\chi)=L$
where $\chi=(\chi_{k_1}\cdots\chi_{k_{p'}},\gamma)$. 

Thus $\lambda_P$ is onto and therefore $P$ is a CY bundle over $X$.
\end{proof}

\begin{thm}
Let $X$ be a compact K\"ahler manifold. If $H$ is holomorphically isomorphic to $(\Cx)^a\times\C^b\times G_0$ for some
$a\geq {p'}$ and $b\geq q$, where ${p'},q$ are defined as above. Then there exists a CY $H$-bundle over $X$ whose character map is onto.
\end{thm}
\begin{proof}
The preceding proposition tells us that $P$ defined above is a CY $((\Cx)^{p'}\times\C^q)$-bundle whose character map is onto. Let $$N:=P\oplus\bigg(X\times\big((\Cx)^{a-{p'}}\times\C^{b-q}\times G_0\big) \bigg).$$
Then $N$ is a principal $((\Cx)^a\times\C^b\times G_0)$-bundle.

Since $X\times((\Cx)^{a-{p'}}\times\C^{b-q}\times G_0)$ is the trivial $((\Cx)^{a-{p'}}\times\C^{b-q}\times G_0)$-bundle over $X$, its character map is trivial.
It follows that the image of the character map of $N$ is the same as that of $P$, 
thus $N$ is a CY $((\Cx)^a\times\C^b\times G_0)$-bundle whose character map is onto.

Let $\psi: (\Cx)^a\times\C^b\times G_0\ra H$ be a holomorphic isomorphism. 
Let $M:=N\times_{(\Cx)^a\times\C^b\times G_0}H$ be the principal $H$-bundle induced by $\psi$ and $N$.

The by Proposition \ref{new bundle} the following diagram is commutative
\[\xymatrix{\widehat{H}\ar[r]^{\lambda_M}\ar[d]_{\widehat \psi} &Pic(X)\\
\widehat{(\Cx)^a\times\C^b\times G_0}.\ar[ur]_{\lambda_N}
}\]

Since $\psi$ is an isomorphism, $\widehat\psi$ is an isomorphism as well. Thus the image of $\lambda_M$ is the same as that of $\lambda_N$, i.e. $\lambda_M$ is onto.
\end{proof}

\bs

\begin{rem}
The rigidity theorem is not necessarily true if $H=\C$. 
If we have principal bundles $\C-M_i\ra X$, $i=1,2$ and the following diagram
\[\xymatrix{
\widehat{\C}\simeq\C\ar[r]^{\lambda_{M_1}}\ar[d]^{\simeq}_{\xi}&\Pic(X)\\
\widehat{\C}\simeq\C\ar[ur]_{\lambda_{M_2}}
}\]
commutes for some $\xi\in \Aut (\widehat{\C})\simeq \Aut(\C)$,
then the conclusion that $M_1$ and $M_2$ are isomorphic up to a twist is not necessarily true.

For example if we assume that $\Pic(X)$ is discrete, then both character maps $\lambda_{M_1}$ and $\lambda_{M_2}$ have to be trivial. It implies that no matter what $\xi\in\Aut(\C)$ we choose, the above diagram is always commutative. So the above diagram gives no information about whether $M$ and $N$ are isomorphic or not.
\end{rem}

\bs

\section{Acknowledgement}
Part of this paper is the first author's PhD dissertation research project. She owe thanks to the Brandeis Math Department for years of support during her PhD study as well as the hospitality while visiting after her graduation. She would like to thank Professors A. Mayer, D. Ruberman, A. Huang and S.-T. Yau for helpful discussions. Her research is partially supported by a Special Financial Grant from the China Postdoctoral Science Foundation with grant number 2016T90080.

 B.H. Lian would like to thank S. Hosono, A. Huang and S.-T. Yau for helpful discussions over the course of this work. He also thanks the Tsinghua YMSC for hospitality during his recent visits there.
His research is partially supported by an NSF FRG grant MS-1564405 and a Simons Foundation Collaboration Grant(2015-2019).


\vskip.1in

\noindent\address {\SMALL J. Chen, Yau Mathematical Sciences Center, Tsinghua University, Beijing, China 100084.\\ jychen@math.tsinghua.edu.cn.}


\noindent\address {\SMALL B.H. Lian, Department of Mathematics, Brandeis University, Waltham MA, United States 02454.\\ lian@brandeis.edu.}


\begin{thebibliography}{99}

\bibitem[AK]{AK} Abe, Yukitaka; Kopfermann, Klaus {\it Toroidal groups. Line bundles, cohomology and quasi-abelian varieties.} Lecture Notes in Mathematics, 1759. Springer-Verlag, Berlin, 2001.



\bibitem[BHLSY]{BHLSY}  Bloch, Spencer; Huang, An; Lian, Bong H.; Srinivas, Vasudevan; Yau, Shing-Tung {\it On the holonomic rank problem.} J. Differential Geom. 97 (2014), no. 1, 11--35.



\bibitem[Ca]{Ca} Calabi, E. {\it M\'etriques k\"ahl\'eriennes et fibr\'es holomorphes.}
Ann. Sci. Ecole Norm. Sup. (4) 12 (1979), no. 2, 269--294. 

\bibitem[Fu]{Fu} Fulton, William
{\it Algebraic topology.
A first course.} Graduate Texts in Mathematics, 153. Springer-Verlag, New York, 1995.

\bibitem[GH]{GH} Griffiths, Phillip; Harris, Joseph {\it Principles of algebraic geometry.} Reprint of the 1978 original. Wiley Classics Library. John Wiley \& Sons, Inc., New York, 1994.



\bibitem[HLYZ]{HLYZ}  Huang, An; Lian, Bong H.; Yau,  Shing-Tung;  Zhu, Xinwen {\it Chain Integral Solutions to Tautological Systems},  arXiv:1508.00406 [math.AG].  To appear in Math. Res. Lett.

\bibitem[HLZ]{HLZ} Huang, An; Lian, Bong H.; Zhu, Xinwen {\it Period integrals and the Riemann-Hilbert correspondence.} J. Differential Geom. 104 (2016), no. 2, 325--369.


\bibitem[Hu]{Hu} Husemoller, Dale {\it Fibre bundles.} Third edition. Graduate Texts in Mathematics, 20. Springer-Verlag, New York, 1994.

\bibitem[LB]{LB} Birkenhake, Christina; Lange, Herbert {\it Complex abelian varieties.} Second edition. Grundlehren der Mathematischen Wissenschaften [Fundamental Principles of Mathematical Sciences], 302. Springer-Verlag, Berlin, 2004. 

\bibitem[LSY]{LSY}  Lian, Bong H.; Song, Ruifang; Yau, Shing-Tung {\it Periodic integrals and tautological systems.} J. Eur. Math. Soc. (JEMS) 15 (2013), no. 4, 1457--1483.

\bibitem[LY]{LY}  Lian, Bong H.; Yau, Shing-Tung {\it Period integrals of CY and general type complete intersections.} Invent. Math. 191 (2013), no. 1, 35--89.


\bibitem[Mo]{Mo} Morimoto, Akihiko {\it On the classification of noncompact complex abelian Lie groups.} Trans. Amer. Math. Soc. 123 1966 200--228.



\end{thebibliography}
\end{document}